\documentclass[11pt,a4paper]{article}
\usepackage{comment}
\usepackage{calc}
\usepackage{tikz}
\usepackage{epsf,epsfig,amsfonts,amsgen,amsmath,amssymb,amstext,amsbsy,amsopn,amsthm}
\usepackage{color}
\usepackage{graphicx}
\usepackage{caption}
\usepackage{epstopdf}
\usepackage{subfigure}

\newtheorem{thm}{Theorem}

\newtheorem{lem}{Lemma}

\newtheorem{obs}{Observation}
\theoremstyle{definition}

\newtheorem{conj}{Conjecture}
\newtheorem{claim}{Claim}
\newtheorem{case}{Case}
\newtheorem{remark}[claim]{Remark}

\begin{document}

\title
{Rainbow triangles in arc-colored digraphs
\thanks{The first author is supported by GXNSF (Nos. 2016GXNSFFA38001 and 2018GXNSFAA138152) and Program on the High Level  Innovation Team and Outstanding Scholars in Universities of Guangxi Province; the second author is supported by  NSFC (Nos. 11571135 and 11671320)
and the third author is supported by the Fundamental Research Funds for the Central Universities (No. 31020180QD124).}}
\author{\quad Wei Li $^{a, b}$,
\quad Shenggui Zhang $^{b}$\thanks{Corresponding author. E-mail addresses: muyu.yu@163.com (W. Li), sgzhang@nwpu.edu.cn (S. Zhang), liruonan@mail.nwpu.edu.cn (R. Li). },
\quad Ruonan Li $^{b}$\\[2mm]
\small $^{a}$ College of Mathematics and Statistics, Guangxi Normal University, \\
\small Guilin,  541004, China\\
\small $^{b}$ Department of Applied Mathematics, School of Sciences, Northwestern Polytechnical University, \\
\small Xi'an,  710029, China\\}
\date{\today}

\maketitle

\begin{abstract}
Let $D$ be an arc-colored digraph. The arc number $a(D)$ of $D$ is defined as the number of arcs of $D$. The color number $c(D)$ of $D$ is defined as the number of colors assigned to the arcs of $D$. A rainbow triangle in $D$ is a directed triangle in which every pair of arcs have distinct colors. Let $f(D)$ be the smallest integer such that if $c(D)\geq f(D)$, then $D$ contains a rainbow triangle. In this paper we obtain $f(\overleftrightarrow{K}_{n})$ and $f(T_n)$, where $\overleftrightarrow{K}_{n}$ is a complete digraph of order $n$ and $T_n$ is a strongly connected tournament of order $n$. Moreover we characterize the arc-colored complete digraph $\overleftrightarrow{K}_{n}$ with $c(\overleftrightarrow{K}_{n})=f(\overleftrightarrow{K}_{n})-1$ and containing no rainbow triangles. We also prove that an arc-colored digraph $D$ on $n$ vertices contains a rainbow triangle when
$a(D)+c(D)\geq a(\overleftrightarrow{K}_{n})+f(\overleftrightarrow{K}_{n})$, which is a directed extension of the undirected case.


\smallskip

\end{abstract}

\section{Introduction}
In this paper we only consider finite digraphs without loops or multiple arcs. For terminology and notations not defined here, we refer the readers to \cite{Bang-Jensen: 2001} and \cite{Bondy: 2008}.

Let $D=(V, A)$ be a digraph. We use $a(D)$ to denote the number of arcs of $D$. If $uv\in A(D)$, then we say that $u$ {\it dominates } $v$ (or $v$ is {\it dominated} by $u$) and $uv$ is an {\it in-arc} of $v$ (or $uv$ is an {\it out-arc} of $u$). For a vertex $v$ of $D$, the {\it in-neighborhood } $N^{-}_{D}(v)$ of $v$ is the set of vertices dominating $v$, and  the {\it out-neighborhood } $N^{+}_{D}(v)$ of $v$ is the set of vertices dominated by $v$. The {\it in-degree} $d^{-}_{D}(v)$ and {\it out-degree} $d^{+}_{D}(v)$ of $v$ are defined as the cardinality of $N^{-}_{D}(v)$ and $N^{+}_{D}(v)$, respectively. The {\it degree} $d_{D}(v)$ of $v$ is defined as the sum of $d^{-}_{D}(v)$ and $d^{+}_{D}(v)$. 
A {\it complete digraph} is a digraph obtained from a complete graph $K_{n}$ by replacing each edge $xy$ of $K_{n}$ with a pair of arcs $xy$ and
$yx$, denoted by $\overleftrightarrow{K}_{n}$. A {\it complete bipartite digraph} is a digraph obtained from a complete bipartite graph $K_{m, n}$ by replacing each edge $xy$ of $K_{m, n}$ with a pair of arcs $xy$ and
$yx$, denoted by $\overleftrightarrow{K}_{m, n}$. A {\it tournament} is a digraph obtained from a complete graph $K_{n}$ by replacing each edge $xy$ of $K_{n}$ with exactly one of the arcs $xy$ and $yx$. A digraph $D$
is {\it strongly connected} if, for each pair of distinct vertices $x$ and $y$ in $D$, there exists an $(x, y)$-path. The subdigraph of $D$ induced by $S\subseteq V(D)$ is denoted by $D[S]$. An {\it arc-coloring} of $D$ is a mapping $C: A(D)\rightarrow \mathbb{N}$, where $\mathbb{N}$ is the set of natural numbers. We call $D$ an {\it arc-colored digraph} if it is assigned such an arc-coloring $C$. We use $C(D)$ and $c(D)$ (called the {\it color number} of $D$) to denote the set and the number of colors assigned to the arcs of $D$, respectively. If $c(D)=k$, then we call $D$ a {\it $k$-arc-colored digraph}. Let $D$ be an arc-colored digraph and $i$ a color in $C(D)$. We use $D^i$ to denote the arc-colored subdigraph of $D$ induced by all the arcs of color $i$. For a vertex $v\in D$, we use $CN^{-}_{D}(v)$ and $CN^{+}_{D}(v)$ to denote the set of colors assigned to the in-arcs and the out-arcs of $v$, respectively. The {\it color neighbor} $CN_{D}(v)$ of $v$ is defined as $CN_{D}(v)=CN^{-}_{D}(v)\bigcup CN^{+}_{D}(v)$. The {\it in-color degree} $d^{-c}_{D}(v)$ and the {\it out-color degree} $d^{+c}_{D}(v)$ of $v$ are the cardinality of $CN^{-}_{D}(v)$ and $CN^{+}_{D}(v)$, respectively. If there is no ambiguity, we often omit the subscript $D$ in the above notations. A {\it rainbow} digraph is a digraph in which every pair of arcs have distinct colors. A {\it rainbow triangle} is a directed triangle which is rainbow.

The existence of rainbow subgraphs has been widely studied, see the survey papers \cite{Fujita: 2014, X.Li: 2008}. In particular, the existence of rainbow triangles attracts much attention during the past decades. For an edge-colored complete graph $K_n$, Gallai \cite{Gallai: 1967} characterized the coloring structure of $K_n$ containing no rainbow triangles. Gy\'{a}rf\'{a}s and Simonyi \cite{Gyarfas: 2004} showed that each edge-colored $K_n$ with $\Delta^{mon}(K_n)<\frac{2n}{5}$ contains a rainbow triangle and this bound is tight. Fujita et al. \cite{FLZ: 2017} proved that each edge-colored $K_n$ with $\delta^c(K_{n})> \log_{2}n$ contains a rainbow triangle and this bound is tight. For a general edge-colored graph $G$ of order $n$, Li and Wang \cite{Li-Wang: 2012} proved that if $\delta^{c}(G)\geq\frac{\sqrt{7}+1}{6}n$, then $G$ contains a rainbow triangle. Li \cite{H.Li: 2013} and Li et al. \cite{B.Li: 2014} improved the condition to $\delta^{c}(G)>\frac{n}{2}$ independently, and showed that this bound is tight. Li et al. \cite{Li-Ning-Zhang: 2016} further proved that if $G$ is an edge-colored graph of order $n$ satisfying $d^c(u)+d^c(v)\geq n+1$ for every edge $uv\in E(G)$, then it contains a rainbow triangle. In \cite{Li-Zhang-Bai-Li: 2018}, Li et al. gave some maximum monochromatic degree conditions for an arc-colored strongly connected tournament  $T_{n}$ to contain rainbow triangles, and to contain rainbow triangles passing through a given vertex. For more results on rainbow cycles, see \cite{Albert: 1995, Erdos: 1983, Frieze: 1993, Hahn: 1986}.

In this paper, we mainly study the existence of rainbow triangles in arc-colored digraphs. Let $D$ be an arc-colored digraph on $n$ vertices. Sridharan \cite{Sridharan} proved that the maximum number of arcs
among all digraphs of order $n$ with no directed triangles is $\lfloor\frac{n^{2}}{2}\rfloor$. Thus $D$ contains a rainbow triangle if $c(D)\geq\lfloor\frac{n^{2}}{2}\rfloor+1$. This lower bound is sharp by considering the complete bipartite digraph $\overleftrightarrow{K}_{\lfloor\frac{n}{2}\rfloor, \lceil\frac{n}{2}\rceil}$ with arcs assigned pairwise distinct colors.

For an edge-colored graph $G$, we use $e(G)$ and $c(G)$ to denote the number of edges of $G$ and the number of colors assigned to the edges of $G$, respectively. Let $f(G)$ be the smallest integer such that if $c(G)\geq f(G)$, then $G$ contains a rainbow triangle. In \cite{Gyarfas: 2004}, the authors proved that $f(K_{n})=n$. Li et al. \cite{B.Li: 2014} proved that if $e(G)+c(G)\geq \frac{n(n+1)}{2}$, then $G$ contains a rainbow triangle. Note that $\frac{n(n+1)}{2}=\frac{n(n-1)}{2}+n=e(K_{n})+f(K_{n})$. Motivated by this result, we wonder whether an arc-colored digraph $D$ on $n$ vertices contains a rainbow triangle when
$$
a(D)+c(D)\geq a(\overleftrightarrow{K}_{n})+f(\overleftrightarrow{K}_{n}).
$$

First we calculate $f(\overleftrightarrow{K}_{n})$ for $n\geq 3$.

\begin{thm}\label{thm1}
Let $\overleftrightarrow{K}_{n}$ be an arc-colored complete digraph of order $n\geq 3$ and $f(\overleftrightarrow{K}_{n})$ be the smallest integer such that $\overleftrightarrow{K}_{n}$ with $c(\overleftrightarrow{K}_{n})\geq f(\overleftrightarrow{K}_{n})$ contains a rainbow triangle. Then
\[f(\overleftrightarrow{K}_{n})=\begin{cases}
\lfloor\frac{n^{2}}{4}\rfloor+3,&\text{$n=3, 4$};\\
\lfloor\frac{n^{2}}{4}\rfloor+2,&
\text{$n\geq 5$}.
\end{cases}\]
\end{thm}

We also investigate the structure of the arc-colored complete digraphs $\overleftrightarrow{K}_{n}$ with $c(\overleftrightarrow{K}_{n})=f(\overleftrightarrow{K}_{n})-1$ and containing no rainbow triangles.

\begin{thm}\label{thm1+}
Let $\mathcal{G}_n$ be the class of arc-colored complete digraphs of order $n$ such that for each $D\in \mathcal{G}_n$, $c(D)=f(D)-1$ and $D$ contains no rainbow triangles. Then each $D$ in $\mathcal{G}_3$ can be decomposed into two arc-disjoint $2$-arc-colored triangles $\Delta_{1}$ and $\Delta_{2}$ such that $C(\Delta_1)\bigcap C(\Delta_2)=\emptyset$. For each $D$ in $\mathcal{G}_4$, there exists a permutation of the vertex set of $D$, say $v_{1}v_{2}v_{3}v_{4}$, such that
\begin{equation*}
\left\{
\begin{aligned}
C(v_1v_2)&=C(v_2v_3)=C(v_3v_4)=C(v_4v_1)=a,\\
C(v_1v_4)&=C(v_4v_3)=C(v_3v_2)=C(v_2v_1)=b,\\
C(v_1v_3)&=c,\quad C(v_3v_1)=d,\\
C(v_2v_4)&=e,\quad C(v_4v_2)=f,
\end{aligned}
\right.
\end{equation*}
where $a, b, c, d, e, f$ are pairwise distinct colors.

Each $D$ in $\mathcal{G}_5$ belongs to one of the following three types of digraphs:
\begin{itemize}
  \item Type I: There is a vertex $v\in V(D)$ such that all arcs incident to $v$ are colored by a same color $c$, $D-v\in\mathcal{G}_4$ and $c\notin C(D-v)$;
  \item Type II: The vertex set of $D$ can be partitioned into two subsets $\{a_1,a_2\}$ and $\{b_1,b_2,b_3\}$ such that the spanning subdigraph $H$ of $D$ with $A(H)=\{a_ib_j|i=1,2;j=1,2,3\}$ (or $A(H)=\{b_ja_i|i=1,2;j=1,2,3\}$) is rainbow and all arcs in $A(D)\setminus A(H)$ are colored by a same new color;
  \item Type III: The vertex set of $D$ can be partitioned into two subsets $\{a_1,a_2\}$ and $\{b_1,b_2,b_3\}$ such that $C(D[\{a_1,a_2\}])=\{a,b\}$, $D[\{b_1,b_2,b_3\}]\in\mathcal{G}_3$, $C(D[\{b_1,b_2,b_3\}])=\{c,d,e,f\}$ and all arcs between $\{a_1,a_2\}$ and $\{b_1,b_2,b_3\}$ are colored by $g$, where $a,b,c,d,e,f,g$ are pairwise distinct colors.
\end{itemize}
For each $D\in\mathcal{G}_n$, $n\geq 6$, the vertex set of $D$ can be partitioned into two subsets $\{a_1,a_2,\ldots,\\
a_{\lfloor\frac{n}{2}\rfloor}\}$ and $\{b_1,b_2,\ldots,b_{\lceil\frac{n}{2}\rceil}\}$ such that the spanning subdigraph $H$ of $D$ with
$$A(H)=\{a_ib_j|i=1,2,\ldots,\lfloor\frac{n}{2}\rfloor;j=1,2,\ldots,\lceil\frac{n}{2}\rceil\}$$ or
$$A(H)=\{b_ja_i|i=1,2,\ldots,\lfloor\frac{n}{2}\rfloor;j=1,2,\ldots,\lceil\frac{n}{2}\rceil\}$$
 is rainbow and all arcs in $A(D)\setminus A(H)$ are colored by a same new color.
\end{thm}

Furthermore, we study the "$a(D)+c(D)$" condition for the existence of rainbow triangles
 in arc-colored digraphs (not necessarily complete).

\begin{thm}\label{thm2}
Let $D$ be an arc-colored digraph on $n$ vertices. If
\[a(D)+c(D)\geq\begin{cases}
n(n-1)+\lfloor\frac{n^{2}}{4}\rfloor+3,&\text{$n=3, 4$};\\
n(n-1)+\lfloor\frac{n^{2}}{4}\rfloor+2,&
\text{$n\geq 5$ },
\end{cases}\]
then $D$ contains a rainbow triangle.
\end{thm}

\begin{remark}
By the definition of $f(\overleftrightarrow{K}_{n})$ and Theorem \ref{thm1}, we can see that the bound of $a(D)+c(D)$ in Theorem \ref{thm2} is sharp.
\end{remark}

Finally, we give a color number condition for the existence of rainbow triangles in strongly connected tournaments.

\begin{thm}\label{thm3}
Let $D$ be an arc-colored strongly connected tournament on $n$ vertices. If $c(D)\geq \frac{n(n-1)}{2}-n+3$, then $D$ contains a rainbow triangle.
\end{thm}

\begin{remark}
The bound of $c(D)$ in Theorem \ref{thm3} is sharp. Let $D$ be a digraph with vertex set $V=\{v_{1}, v_{2}, \ldots, v_{n}\}$ and arc set $A=\left(\{v_{i}v_{j}|1\leq i< j\leq n\}\setminus\{v_{1}v_{n}\}\right)\bigcup\{v_{n}v_{1}\}$. Then $D$ is a strongly connected tournament. Color all the arcs incident to $v_{1}$ by a same color and color the remaining arcs by pairwise distinct new colors. Then $c(D)=\frac{n(n-1)}{2}-(n-1)+1=\frac{n(n-1)}{2}-n+2$. But there is no rainbow triangle in $D$.
\end{remark}

\section{Proofs of the theorems}

Let $v$ be a vertex in $D$, and $c$ a color in $C(D)$. If all the arcs with color $c$ are incident to $v$, then we
call $c$ a color {\it saturated} by $v$. We use $C^{s}(v)$ to denote the set of colors saturated by $v$ and define $d^{s}(v)=|C^{s}(v)|$. If a color in $C(D)$ is not saturated by $v$, then it is also a color
in $C(D-v)$. This implies that $c(D-v) = c(D)-d^{s}(v)$.

\begin{obs}
Let $D$ be an arc-colored complete digraph. For a vertex $v\in D$, if there are two vertices $u\neq w$ such that $C(uv)\neq C(vw)$ and $C(uv), C(vw)\in C^{s}(v)$, then $uvwu$ is a rainbow triangle.
\end{obs}
\begin{proof}
Since the arc $wu$ is not incident to $v$, we have $C(wu)\notin C^{s}(v)$. Namely, $C(uv)$, $C(vw)$ and $C(wu)$ are pairwise distinct colors. Thus, $uvwu$ is a rainbow triangle.
\end{proof}
Before presenting the proof of Theorem \ref{thm1}, we first prove the following lemmas.

\begin{lem}\label{lem1}
Let $D$ be an arc-colored digraph of order $n\geq 4$ without rainbow triangles. For a vertex $v\in D$, if $D-v\cong \overleftrightarrow{K}_{n-1}$ and $d^{s}(v)\geq 3$, then $CN^{-}(v)\bigcap C^{s}(v)=\emptyset$ or $CN^{+}(v)\bigcap C^{s}(v)=\emptyset$. Moreover, if $D\cong \overleftrightarrow{K}_4$, then $c(D)\leq 5$.
\end{lem}

\begin{proof}
Let $C^{s}(v)=\{1,2, \ldots, k\}$, $k\geq 3$. If $CN^{+}(v)\bigcap C^{s}(v)\neq\emptyset$, without loss of generality, assume that $C(vw)=1$. We will show that $CN^{-}(v)\bigcap C^{s}(v)=\emptyset$. By contradiction, assume that there is a vertex $u\neq w$ such that $C(uv)\in C^{s}(v)\setminus \{1\}$, then $uvwu$ is a rainbow triangle, a contradiction. Suppose that $C(wv)\in C^{s}(v)\setminus \{1\}$, such as $C(wv)=2$. Since $d^{s}(v)\geq 3$, there is an arc colored by $3$ incident to $v$. By the above argument, this arc must be an out-arc of $v$, so we can assume that $C(vu)=3$. But now $wvuw$ is a rainbow triangle, a contradiction. Thus $CN^{-}(v)\bigcap \left(C^{s}(v)\setminus \{1\}\right)=\emptyset$. Namely, $2\in CN^{+}(v)$, by similar analysis we have $CN^{-}(v)\bigcap \left(C^{s}(v)\setminus \{2\}\right)=\emptyset$. Hence, we have $CN^{-}(v)\bigcap C^{s}(v)=\emptyset$.

If $n=4$, then assume that $V(D)=\{v, x, y, z\}$, $\{1, 2, 3\}\subseteq C^{s}(v)$, $C(vx)=1$, $C(vy)=2$, $C(vz)=3$, $D[\{x,y,z\}]$ is a $\overleftrightarrow{K}_3$, $C(xv)=a$, $C(yv)=b$ and $C(zv)=c$. Since $D$ contains no rainbow triangles, we have $C(yx)=C(zx)=a$, $C(xy)=C(zy)=b$ and $C(xz)=C(yz)=c$. So $C(D)=\{1, 2, 3\} \cup \{a\}\cup\{b\}\cup\{c\}$. If $a$, $b$, $c$ are pairwise distinct, then $xyzx$ is a rainbow triangle, a contradiction. So two of $a$, $b$ and $c$ must be a same color. Then $c(D)\leq 5$.
\end{proof}

\begin{lem}\label{lem2}
Let $D$ be an arc-colored complete digraph of order $n\geq 4$. If $D$ contains no rainbow triangles, then there must be a vertex $v \in V(D)$ such that $d^{s}(v)\leq \lfloor\frac{n}{2}\rfloor$.
\end{lem}

\begin{proof}
Suppose for every vertex $v\in V(D)$, we have $d^{s}(v)\geq \lfloor\frac{n}{2}\rfloor+1$. Let $v$ be a vertex of $D$. Since $n\geq 4$, we have $d^{s}(v)\geq 3$. By Lemma \ref{lem1}, either $CN^{-}(v)\bigcap C^{s}(v)=\emptyset$ or $CN^{+}(v)\bigcap C^{s}(v)=\emptyset$. Without loss of generality, suppose $C^{s}(v)=\{1, 2, \ldots, k\}$, $k\geq \lfloor\frac{n}{2}\rfloor+1$ and $C(vw_{i})=i$, for $i=1, \ldots, k$. For $j\neq 1$, since $D$ contains no rainbow triangles, $C(w_{1}w_{j})\neq 1$ and $C(w_{j}v)\neq 1$, we have $C(w_{1}w_{j})=C(w_{j}v)$. Thus, $C(w_{1}w_{j}) \notin C^{s}(w_{1})$ for $j=2, \ldots, k$. Similarly, for $j\neq 1$, since $D$ contains no rainbow triangles, $C(w_{j}w_{1})\neq j$ and $C(w_{1}v)\neq j$, we have $C(w_{j}w_{1})=C(w_{1}v)$. Since $C(w_{j}w_{1})\in CN^{-}(w_{1})\bigcap CN^{+}(w_{1})$, we can see that $C(w_{j}w_{1}) \notin C^{s}(w_{1})$ for $j=2, \ldots, k$. So, all colors assigned to the arcs between $w_{1}$ and $\{w_{2}, \ldots, w_{k}\}$ do not belong to $C^{s}(w_{1})$. Note that for a pair of arcs $uw_{1}$ and $w_{1}u$, at most one of them has a color in $C^{s}(w_{1})$. So,
$$|V(D-\{w_{1}, \ldots, w_{k}\})|\geq d^{s}(w_{1})\geq \lfloor\frac{n}{2}\rfloor+1.$$ But now
$$|V(D)|=n\geq \lfloor\frac{n}{2}\rfloor+1+\lfloor\frac{n}{2}\rfloor+1\geq n+1,$$ a contradiction.
\end{proof}

Now we can give the proof of Theorem \ref{thm1}.

\noindent\textbf{Proof of Theorem \ref{thm1}.}
We divide the proof into four cases.
\begin{case}
$n=3$.
\end{case}
If $n=3$, then we have $a(D)=6$. If $c(D)\geq \lfloor\frac{n^{2}}{4}\rfloor+3=5$, then at most two arcs have a same color, other arcs all have pairwise distinct new colors. Since there are two arc-disjoint triangles in $D$, at least one of them is rainbow. Let $V(D)=\{u, v, w\}$, $C(uv)=C(vw)=1$, $C(wu)=2$, $C(vu)=C(uw)=3$ and $C(wv)=4$. Then $c(D)=4$ and neither of two triangles are rainbow. So we have $f(\overleftrightarrow{K}_{3})=\lfloor\frac{n^{2}}{4}\rfloor+3=5$.
\begin{case}
$n=4$.
\end{case}
For $n=4$, if $c(D)\geq \lfloor\frac{n^{2}}{4}\rfloor+3=7$ but $D$ contains no rainbow triangles, then for every vertex $v \in V(D)$, the complete digraph $D-v$ contains no rainbow triangles either. Since $f(\overleftrightarrow{K}_{3})=5$, we have $c(D-v)=c(D)-d^{s}(v)\leq 4$. So for every vertex $v \in V(D)$, we have $d^{s}(v)\geq 3$. By Lemma \ref{lem1}, we have $c(D)\leq 5$, a contradiction.

Let $V(D)=\{v, x, y, z\}$, $C(vx)=C(xy)=C(yz)=C(zv)=1$, $C(vz)=C(zy)=C(yx)=C(xv)=2$, $C(vy)=3$, $C(yv)=4$, $C(xz)=5$ and $C(zx)=6$. Then $c(D)=6$ and $D$ contains no rainbow triangles. So we have $f(\overleftrightarrow{K}_{4})=\lfloor\frac{n^{2}}{4}\rfloor+3=7$.

\setcounter{claim}{0}
\begin{claim}\label{claim1}
Let $D$ be an arc-colored $\overleftrightarrow{K}_{4}$ without rainbow triangles. If $c(D)=6$, then there must be a permutation of the vertex set of $D$, say $v_{1}v_{2}v_{3}v_{4}$, such that
\begin{equation*}
\left\{
\begin{aligned}
C(v_1v_2)&=C(v_2v_3)=C(v_3v_4)=C(v_4v_1)=a,\\
C(v_1v_4)&=C(v_4v_3)=C(v_3v_2)=C(v_2v_1)=b,\\
C(v_1v_3)&=c,\quad C(v_3v_1)=d,\\
C(v_2v_4)&=e,\quad C(v_4v_2)=f,
\end{aligned}
\right.
\end{equation*}
where $a, b, c, d, e, f$ are pairwise distinct colors.
\end{claim}
\begin{proof}
Since $f(\overleftrightarrow{K}_{3})=5$, $c(D)=6$ and $D$ contains no rainbow triangles, we have $d^s(v)\geq 2$ for each vertex $v\in V(D)$. If there is a vertex $v\in V(D)$ such that $d^{s}(v)\geq 3$, then by Lemma \ref{lem1}, we have $c(D)\leq 5$, a contradiction. So we have $d^{s}(v)= 2$ for every vertex $v \in V(D)$. Thus for each $v\in V(D)$, $D-v$ belongs to $\mathcal{G}_3$.

By the structure of $\mathcal{G}_3$, we know that for each color $i$ the arc-colored digraph $D^i$ must be connected (otherwise, we recolor a component of $D^i$ by a new color, then the obtained arc-colored complete digraph has $f(\overleftrightarrow{K}_{4})$ colors but contains no rainbow triangles, a contradiction) and belong to one of the following four types.
~\\
Type 1: an arc;\\
Type 2: a directed path of length 2;\\
Type 3: a directed path of length 3;\\
Type 4: a directed cycle of length 4.

Let $X_j=\{i\in C(D): \text{ $D^i$ belongs to Type $j$}\}$ and $x_j=|X_j|$ for $j=1, 2, 3, 4$.
Then we have
\begin{equation*}
\left\{                         
\begin{aligned}
&x_1+x_2+x_3+x_4=c(D)\\
&x_1+2x_2+3x_3+4x_4=a(D)\\
&x_2+2x_3+4x_4=2\binom{4}{3}~~~\text{(the number of directed triangles in $D$})\\
&\text{$x_j\in \mathbb{N}$ for $j=1,2,3,4.$}
\end{aligned}
\right.
\end{equation*}
By these equations, we get $x_1=4, x_2=x_3=0$ and $x_4=2$.
Without loss of generality, let $X_1=\{1,2,3,4\}$,  $X_4=\{5,6\}$  and let $uxyzu$ be the directed cycle of length $4$ colored by $5$. If $C(xu)\in X_1$, then $C(yx),C(uz)\not\in X_1$ (otherwise, $yxuy$ or $xuzx$ is a rainbow triangle). This forces $C(yx)=C(uz)=6$. Note that $D^6$ is a directed cycle of length $4$. We have  $C(xu)=6\in X_4$. This contradicts to the assumption that  $C(xu)\in X_1$. Thus $C(xu)\not\in X_1$. This forces $C(xu)=6$. By the symmetry of the cycle $uxyzu$, we get $C(xu)=C(uz)=C(zy)=C(yx)=6$. For each color $i=1,2,3,4$, $D^i$ is an arc.
\end{proof}

Let $D$ be an arc-colored complete digraph of order $n\geq 5$ with vertex set $\{v_{1}, v_{2}, \ldots, v_{n}\}$. Let $$R=\{v_{2i-1}v_{2j}|i=1, 2, \ldots, \lceil\frac{n}{2}\rceil, j=1, 2, \ldots, \lfloor\frac{n}{2}\rfloor\}.$$Color the arcs in $R$ with pairwise distinct colors and color the remaining arcs with a same new color. Then $c(D)=\lfloor\frac{n^{2}}{4}\rfloor+1$ and $D$ contains no rainbow triangles. So $f(\overleftrightarrow{K}_{n})\geq\lfloor\frac{n^{2}}{4}\rfloor+2$ for $n\geq 5$.
\begin{case}
$n=5$.
\end{case}
For $n=5$, if $c(D)\geq \lfloor\frac{n^{2}}{4}\rfloor+2=8$ but $D$ contains no rainbow triangles, then for every vertex $v \in V(D)$, the complete digraph $D-v$ contains no rainbow triangles either. Since $f(\overleftrightarrow{K}_{4})=7$, we have $c(D-v)=c(D)-d^{s}(v)\leq 6$. So for every vertex $v \in V(D)$, we have $d^{s}(v)\geq 2$. On the other hand, by Lemma \ref{lem2}, there must be a vertex $v \in V(D)$ such that $d^{s}(v)\leq \lfloor\frac{n}{2}\rfloor=2$. So there exists a vertex $v \in V(D)$ such that $d^{s}(v)=2$. Let $D'=D-v$, then $D'$ is an arc-colored $\overleftrightarrow{K}_{4}$ without rainbow triangles and $c(D')=6$. By Claim \ref{claim1}, we can assume that $V(D')=\{u, x, y, z\}$ and
\begin{equation*}
\left\{
\begin{aligned}
C(ux)&=C(xy)=C(yz)=C(zu)=5,\\
C(uz)&=C(zy)=C(yx)=C(xu)=6,\\
C(uy)&=1,\quad C(yu)=2,\\
C(xz)&=3,\quad C(zx)=4.
\end{aligned}
\right.
\end{equation*}
Let $C^{s}(v)=\{7, 8\}$. Without loss of generality, we can assume that $C(vu)=7$. Considering the triangle $vuxv$, we have $C(xv)\neq 8$. If $C(vx)=8$, then considering the triangles $vuzv$ and $vxzv$, we have $C(zv)\in \{6,7\}\bigcap\{3,8\}$, a contradiction. So $C(vx)\neq 8$. Similarly, we have
$$8\notin\{C(vy)\}\cup\{C(yv)\}\cup\{C(vz)\}\cup\{C(zv)\}.$$
So $C(uv)=8$. By similar analysis, we have
$$7\notin\{C(vx)\}\cup\{C(xv)\}\cup\{C(vy)\}\cup\{C(yv)\}\cup\{C(vz)\}\cup\{C(zv)\}.$$
Considering the triangles $vuxv$ and $vyuv$, we have $C(xv)=5$ and $C(vy)=2$. But now $xvyx$ is a rainbow triangle, a contradiction.
Thus, we have $f(\overleftrightarrow{K}_{5})=\lfloor\frac{n^{2}}{4}\rfloor+2=8$.
\begin{case}
$n\geq 6$.
\end{case}
Suppose Theorem \ref{thm1} is true for $\overleftrightarrow{K}_{n-1}$, now we consider $\overleftrightarrow{K}_{n}$, $n\geq 6$.
Let $D$ be an arc-colored complete digraph of order $n\geq 6$. If $c(D)\geq \lfloor\frac{n^{2}}{4}\rfloor+2$ but $D$ contains no rainbow triangles, then for every vertex $v \in V(D)$, the digraph $D-v$ contains no rainbow triangles either. Thus, we have $c(D-v)=c(D)-d^{s}(v)\leq \lfloor\frac{(n-1)^{2}}{4}\rfloor+1$. So for every vertex $v \in V(D)$, we have
$$d^{s}(v)\geq \lfloor\frac{n^{2}}{4}\rfloor+2-\left(\lfloor\frac{(n-1)^{2}}{4}\rfloor+1\right)$$
\begin{equation*}
=
\left\{
\begin{aligned}
&\frac{n}{2}+1, &n \text{~is~even};\\
&\frac{n+1}{2}, &n \text{~is~odd}.
\end{aligned}
\right.
\end{equation*}
On the other hand, by Lemma \ref{lem2}, there must be a vertex $v \in V(D)$ such that $d^{s}(v)\leq \lfloor\frac{n}{2}\rfloor$, a contradiction. So we have $f(\overleftrightarrow{K}_{n})=\lfloor\frac{n^{2}}{4}\rfloor+2$ for $n\geq 5$.

The proof is complete. \qed

\noindent\textbf{Proof of Theorem \ref{thm1+}.}
Let $D\in \mathcal{G}_3$.
Since the two arc-disjoint directed triangles $\Delta_1$ and $\Delta_2$ are not rainbow, we have $c(\Delta_1)\leq 2$ and $c(\Delta_2)\leq 2$. Thus $4=c(D)\leq c(\Delta_1)+c(\Delta_2)\leq 4$, the equality holds if and only if $c(\Delta_1)=c(\Delta_2)=2$ and $C(\Delta_1)\bigcap C(\Delta_2)=\emptyset$.

We have already characterized $\mathcal{G}_4$ by Claim \ref{claim1} in Theorem \ref{thm1}.

Let $D\in\mathcal{G}_5$. Then $c(D)=f(\overleftrightarrow{K}_{5})-1=8-1=7$. If $d^s(v)=1$ for a vertex $v\in V(D)$, then we have $c(D-v)=6$. By Claim \ref{claim1} in Theorem \ref{thm1}, we can assume that $V(D-v)=\{u, x, y, z\}$ and
\begin{equation*}
\left\{
\begin{aligned}
C(ux)&=C(xy)=C(yz)=C(zu)=5,\\
C(uz)&=C(zy)=C(yx)=C(xu)=6,\\
C(uy)&=1,\quad C(yu)=2,\\
C(xz)&=3,\quad C(zx)=4.
\end{aligned}
\right.
\end{equation*}
Without loss of generality, we can assume that $C(vu)=7\in C^{s}(v)$. Considering triangles $vuxv$, $vuyv$ and $vuzv$, we have $C(xv)=5$ or $7$, $C(yv)=1$ or $7$ and $C(zv)=6$ or $7$. Considering triangles $vzxv$ and $vzyv$, we  have $C(vz)\in\{4,5,7\}\bigcap\{1,6,7\}$, and hence $C(vz)=7=C(xv)=C(yv)$. Considering triangles $vxzv$ and $vxyv$, we have $C(vx)\in\{5,7\}\bigcap\{3,6,7\}$, and hence $C(vx)=7=C(zv)$. Considering triangles $vyzv$ and $vyxv$, we have $C(vy)\in\{5,7\}\bigcap\{6,7\}$, and hence $C(vy)=7$. Finally, considering triangles $uvxu$ and $uvyu$, we have $C(uv)\in\{6,7\}\bigcap\{2,7\}$, and hence $C(uv)=7$. Thus, all arcs incident to $v$ are colored by $7$ and $D$ belongs to Type I.

	Now let us consider the case that $d^s(v)\geq 2$ for each vertex $v\in V(D)$. Let
	$$X=\{i\in C(D): C(uv)=i \text{~and~} i\in C^s(u)\cap C^s(v)\},$$
    $$Y=\{i\in C(D): i\in C^s(v) \text{~and~} i\not\in C^s(u) \text{~if~} u\neq v\},$$
    $$Z=\{i\in C(D): i\not\in C^s(v) \text{~for any~} v\in V(D)\}.$$
    Let $x,y$ and $z$ be the cardinality of $X, Y$ and $Z$, respectively. Then we have
    \begin{equation*}
    \left\{                         
    \begin{aligned}
    &x+y+z=c(D)\\
    &2x+y=\sum_{v\in V(D)}d^s(v).\\
    \end{aligned}
    \right.
    \end{equation*}
    Recall that $c(D)=7$ and $d^s(v)\geq 2$ for each vertex $v\in V(D)$. We have
        \begin{equation*}
    \left\{                         
    \begin{aligned}
    &x+y+z=7\\
    &2x+y\geq 10.\\
    \end{aligned}
    \right.
    \end{equation*}
    Thus $x\geq z+3\geq 3$.

    Let $H$ be an arc-colored spanning subdigraph of $D$ with the arcs that are assigned colors in $X$. Then $a(H)\geq x\geq 3$. Since each directed path $uvw$ in $H$ implies a rainbow triangle $uvwu$, there is no directed path of length $2$ in $H$. Let $\hat{H}$ be the underlying graph of $H$.
    \setcounter{case}{0}
      \begin{case}
    	$Z=\emptyset$.
    \end{case}
    If $uv,ab\in A(H)$ for four distinct vertices $u,v,a,b$,  then without loss of generality, we can assume that $C(va)=1$. Since $vauv$ and $abva$ are not rainbow triangles, it is easy to see that $C(au)=C(bv)=1$. Thus $1\in Z$. This contradicts that $Z=\emptyset$. Thus there exists a vertex $u$ such that each arc in $H$ is incident to $u$ and hence $d^s(u)\geq x$.


    Let $uv$ be an arc such that $C(uv)\in X$. Let $\{a,b,c\}=V(D)\backslash\{u,v\}$. Since $D$ contains no rainbow triangles and $Z=\emptyset$, we can assume that $C(va)=C(au)=1$, $C(vb)=C(bu)=2$ and $C(vc)=C(cu)=3$. This implies that  $\{1,2,3\}\subseteq Y$ and $y\geq 3$.
    Now we have $x,y\geq 3$ and $x+y=7$. Thus either $x=3,y=4$ or $x=4,y=3$.

    If $x=3,y=4$, then $10=2x+y=\sum_{v\in V(D)\setminus\{u\}}d^s(v)+d^s(u)\geq 11$, a contradiction.

    If $x=4,y=3$, then $11=2x+y=\sum_{v\in V(D)\setminus\{u\}}d^s(v)+d^s(u)\geq 12$, a contradiction.
%
%
    \begin{case}
    	$x\geq 5$.
    \end{case}
    If $H$ contains a cycle $uvu$, namely, $C(uv),C(vu)\in X$, where $v\neq u$, then it is easy to see that none of the arcs in $H$ appears between $\{u,v\}$ and $V(D)\backslash\{u,v\}$. Let $\{a,b,c\}=V(D)\backslash\{u,v\}$. Then either the triangle $abca$ or the triangle $cbac$ contains two arcs of $H$. In both cases, we get a rainbow triangle. So $H$ contains no two oppositely oriented arcs.
    Moreover, there is no odd cycle in $\hat{H}$ (otherwise, there must be a directed path of length $2$ in $H$, a contradiction.)

    Note that $a(H)\geq x\geq 5$. The graph $\hat{H}$ must contain a cycle, which has to be of length $4$, say  $a_1b_1a_2b_2a_1$. Let $\{u\}=V(D)\backslash\{a_1,a_2,b_1,b_2\}$. Since there is no directed path of length $2$ in $H$, we can assume that $a_1b_1, a_1b_2, a_2b_1, a_2b_2\in A(H)$ and all the other arcs in $D[a_1,a_2,b_1,b_2]$ are not contained in $H$.
    Assume that $C(a_1a_2)=1$. Then $1\in Y\cup Z$. Consider triangles $a_1a_2b_1a_1$ and $a_1a_2b_2a_1$. We get $C(b_1a_1)=C(b_2a_1)=1$. Consider $b_2a_1b_1b_2$ and $b_1a_1b_2b_1$. We get $C(b_1b_2)=C(b_2b_1)=1$. By similar analyzing process, we finally see that all the arcs in $A(D-u)\backslash A(H)$ are of color $1$.
    Recall that $a(H)\geq 5$. By the symmetry, we can assume that $u$ is either incident to $a_1$ or $b_1$ in $H$.

    If $u$ is incident to $b_1$ in $H$, then the situation has to be $ub_1\in H$ (since $H$ contains no path of length $2$). Now consider triangles $ub_1a_1u$, $ub_1b_2u$ and $ub_1a_2u$. We get $C(a_1u)=C(b_2u)=C(a_2u)=1$. Consider triangles $ua_1b_2u$ and $ua_2b_2u$. We get $C(ua_1)=C(ua_2)=1$. Again, consider the triangle $ua_1b_1u$. We get $C(b_1u)=1$. Now $c(D-ub_2)=6$. Since $c(D)=7$, there holds $C(ub_2)\not\in C(D-ub_2)$. Thus $ub_2\in A(H)$. If $u$ is incident to $a_1$ in $H$, then by a similar analyzing process, we can obtain that $a_1u,a_2u\in A(H)$ and all the other arcs incident to $u$ are of color 1.

    In summary, $H$ is an orientation of $K_{2,3}$ with partite sets $A$ and $B$ such that $|A|=2,|B|=3$ and all the arcs are from $A$ to $B$ or from $B$ to $A$. The remaining arcs in $D$ are all colored by a same new color. So $D$ belongs to Type II.
\begin{case}
	$z\geq 1$ and $x\leq 4$.
\end{case}
    Recall that $x\geq z+3\geq 4$ and $x+y+z=7$. We have $x=4$, $y=2$ and $z=1$. Note that
    $$10=2x+y= \sum_{v\in V(D)} d^s(v)\geq 2*5=10.$$
    So $d^s(v)=2$ for each vertex $v\in V(D)$. Now we assert that $d^s(v)\geq d_{\hat{H}}(v)$. If each color in $C^s(v)\cap X$ is only assigned to one arc in $D$, then there is nothing to prove. If there is a color in  $C^s(v)\cap X$ assigned to more than two arcs, then by the definition of $X$, we know that these arcs must be two oppositely oriented arcs, say $vw$ and $wv$. Recolor $vw$ by a new color. Then the obtained arc-colored complete digraph $D'$ satisfies that $c(D')=f(\overleftrightarrow{K}_{5})$ but $D'$ contains no rainbow triangles, a contradiction. So we have $d^s(v)\geq d_{\hat{H}}(v)$ for each vertex $v\in V(D)$.
   Thus the maximum degree of $\hat{H}$ is at most 2.

   If $\hat{H}$ contains a path of length 3, then without loss of generality, we can assume that $uvws$ is a path with $uv, wv, ws\in A(H)$. Let $C(vu)=1$, $C(vw)=c$ and let $p$ be the vertex in $D$ different from $u,v,w$ and $s$. Since $D$ contains no rainbow triangles, we obtain that $C(uw)=C(su)=C(vs)=C(sw)=1$ and $C(wu)=C(sv)=c$. It is easy to observe that $1,c\in Z$. Since $z=1$, we have $c=1$ and $Z=\{1\}$. Consider triangles $uvpu$, $wvpw$ and $wspw$. We get $C(vp)=C(pu)=C(pw)=C(sp)=a$. If $a\in C^{s}(p)$, then considering triangles $upwu$, $pvsp$, $wpuw$ and $vpsv$, we can get $$\{C(up)\}\cup\{C(pv)\}\cup\{C(wp)\}\cup\{C(ps)\}\subseteq\{1,a\}.$$ Thus, we have $d^{s}(p)=1$, a contradiction. So $a\notin C^{s}(p)$ and hence $a=1$. If $C(up)=2\in C^{s}(p)$, then considering triangles $vupv$ and $sups$, we can get
   $C(pv), C(ps)\in \{1,2\}.$ Since $d^{s}(p)=2$, we have $C(wp)\neq 2$ and $C(wp)\in C^{s}(p)$. Let $C(wp)=3$. Consider triangles $wpvw$ and $wpsw$. We can get $C(pv), C(ps)\in \{1,3\}$. So $C(pv)=C(ps)=1$. But now $\{2,3,C(uv),C(wv),C(ws)\}\subseteq X$. This contradicts that $x=4$. Thus $C(up)\notin C^{s}(p)$. By similar analyzing process, we can see that $C(pv),C(wp),C(ps)\notin C^{s}(p)$. This implies that $C^{s}(p)=\emptyset$, a contradiction.

    If the longest path in $\hat{H}$ is of length $1$, then the arcs of $H$ form two vertex-disjoint cycles of length 2, say $A(H)=\{uv,vu,pq,qp\}$. Since $z=1$, it is easy to check that all the arcs between $\{u,v\}$ and $\{p,q\}$ has a same color, namely, the unique color in $Z$. Let $Y=\{1,2\}$ and $V(D)\backslash\{u,v,p,q\}=\{w\}$. Then there holds $C^s(w)=\{1,2\}$. Since $D$ contains no rainbow triangles, we have $C(uw)=C(wv), C(vw)=C(wu), C(pw)=C(wq), C(qw)=C(wp)$. By the symmetry, we can assume that $C(uw)=C(wv)=1$. Consider triangles $uwpu,uwqu, pwvp$ and $qwvq$. We can see that the color 2 does not appear between $w$ and $\{p,q\}$. This forces $C(vw)=C(wu)=2$ and all the arcs between $w$ and $\{p,q\}$ are colored by the unique color in $Z$. So $D$ belongs to Type III.

    The remaining case is that $\hat{H}$ is composed of a path of length 2 and a cycle of length 2. Let $V(D)=\{u,v,w,p,q\}$ and $A(H)=\{uv,wv, pq,qp\}$. Assume that $C(vp)=a$ and $C(pv)=b$. Then it is easy to check that each arcs between $\{u,v,w\}$ and $\{p,q\}$ are of color $a$ or $b$, and $a,b\in Z$. This forces $a=b$ (since $z=1$). Now the arcs $vu,uw,wu,vw$ are the only possible arcs that are assigned the colors in $Y$. Thus $c(D[v,u,w])=4$ and each color in $D[v,u,w]$ does not appears on $A(D)\backslash A(D[v,u,w])$. So $D[v,u,w]\in \mathcal{G}_3$ and $D$ belongs to Type III.

Let $D\in \mathcal{G}_6$.
Since $D$ contains no rainbow triangles, we have $c(D-v)\leq 7$, so $d^{s}(v)\geq 3$ for every $v\in V(D)$. On the other hand, by Lemma \ref{lem2}, there is a vertex $v\in V(D)$ such that $d^{s}(v)\leq \lfloor\frac{n}{2}\rfloor=3$. So there is a vertex $v\in V(D)$ such that $d^{s}(v)=3$ and $c(D-v)=7$. Since $D-v$ contains no rainbow triangles, by the above arguments, $D-v\in\mathcal{G}_5$ and thus belongs to one of the three types of digraphs.

\setcounter{case}{0}
\begin{case}
$D-v$ belongs to Type I.
\end{case}
Let $V(D)=\{u,v,w,x,y,z\}$, We can assume that
\begin{equation*}
\left\{
\begin{aligned}
C(uy)&=1,\quad C(yu)=2,\quad
C(xz)=3,\quad C(zx)=4,\\
C(ux)&=C(xy)=C(yz)=C(zu)=5,\\
C(uz)&=C(zy)=C(yx)=C(xu)=6,\\
C(wx)&=C(wy)=C(wz)=C(wu)=C(uw)=C(xw)=C(yw)=C(zw)=7.
\end{aligned}
\right.
\end{equation*}
Since $d^{s}(v)=3$, by Lemma \ref{lem1}, we can assume that $CN^{-}(v)\bigcap C^{s}(v)=\emptyset$. Let $C^{s}(v)=\{8,9,10\}$ and $C(vz)=8$. Considering the triangle $vzwv$, we have $C(wv)=7$. But now $d^{s}(w)\leq 2$, a contradiction.

\begin{case}
$D-v$ belongs to Type III.
\end{case}
Let $V(D)=\{v,a_1,a_2,b_1,b_2,b_3\}$. We can assume that $C(a_1a_2)=1$, $C(a_2a_1)=2$ and $C(\{a_1,a_2\},\{b_1,b_2,b_3\})=\{3\}$. Since $d^{s}(v)=3$, by Lemma \ref{lem1}, we can assume that $CN^{-}(v)\bigcap C^{s}(v)=\emptyset$. Then there must be a vertex $b_j$ such that $C(vb_j)\in C^{s}(v)$. Without loss of generality, we can assume that $C(vb_1)=8\in C^{s}(v)$. Considering triangles $vb_1a_1v$ and $vb_1a_2v$, we have $C(a_1v)=C(a_2v)=3$. Considering the triangle $a_1va_2a_1$, we have $C(va_2)\in \{2,3\}$. Since $C(a_1b_1)=3$, we can see that $3\notin C^{s}(a_2)$ and $d^{s}(a_2)\leq 2$, a contradiction.

\begin{case}
$D-v$ belongs to Type II.
\end{case}
Let $V(D)=\{v,a_1,a_2,b_1,b_2,b_3\}$. We can assume that
\begin{equation*}
\left\{
\begin{aligned}
C(a_1b_1)&=1,\quad C(a_1b_2)=2,\quad
C(a_1b_3)=3,\\
C(a_2b_1)&=4,\quad C(a_2b_2)=5, \quad C(a_2b_3))=6,
\end{aligned}
\right.
\end{equation*}
and the remaining arcs of $D-v$ are all colored by $7$.

\noindent\textbf{Case 3.1.} $CN^{+}(v)\bigcap C^{s}(v)=\emptyset$.

Since $d^{s}(v)=3$, there must be a vertex $b_j$ such that $C(b_jv)\in C^{s}(v)$. Without loss of generality, we can assume that $C(b_1v)=8\in C^{s}(v)$. Considering triangles $va_1b_1v$ and $va_2b_1v$, we have $C(va_1)=1$ and $C(va_2)=4$. Considering triangles $a_1b_2va_1$ and $a_2b_2va_2$, we have $C(b_2v)\in\{1,2\}\bigcap\{4,5\}$, a contradiction.

\noindent\textbf{Case 3.2.} $CN^{-}(v)\bigcap C^{s}(v)=\emptyset$.

Let $C^{s}(v)=\{8,9,10\}$. If $C(va_{1})=8$, then considering triangles $va_1b_1v$ and $va_1b_2v$, we have $C(b_1v)=1$ and $C(b_2v)=2$. Considering triangles $a_2b_1va_2$ and $a_2b_2va_2$, we have $C(va_2)\in\{1,4\}\bigcap\{2,5\}$, a contradiction. So $C(va_{1})\neq 8$. Similarly we can prove that
$$\left(\{C(va_1)\}\cup\{C(va_2)\}\right)\bigcap C^{s}(v)=\emptyset.$$
Thus $C^{s}(v)\subseteq \{C(vb_1),C(vb_2),C(vb_3)\}$. Without loss of generality, we can assume that $C(vb_1)=8$, $C(vb_2)=9$ and $C(vb_3)=10$. Considering the triangle set
$$
\left\{vb_1uv|u\in\{a_1,a_2,b_2,b_3\}\right\}\bigcup\{vb_2b_1v\},
$$
we have
$$
C(\{uv|u\in\{a_1,a_2,b_1,b_2,b_3\}\})=\{7\}.
$$
Considering triangles $va_1b_1v$, $va_1b_2v$, $va_2b_1v$ and $va_2b_2v$, we have
$$
C(va_1)\in \{1,7\}\bigcap\{2,7\}=\{7\}~\text{and}~C(va_2)\in \{4,7\}\bigcap\{5,7\}=\{7\}.
$$
Let $v=a_3$. Then we can see that the spanning subdigraph $H$ of $D$ with $A(H)=\{a_ib_j|i=1,2,3;j=1,2,3\}$ is rainbow and all the remaining arcs are colored by a same new color $7$. So the theorem is true for $3\leq n\leq 6$.

Let $D\in \mathcal{G}_n$, $n\geq 7$.
Suppose the theorem is true for $\overleftrightarrow{K}_{n-1}$. Now we consider $\overleftrightarrow{K}_{n}$, $n\geq 7$.

If $D=\overleftrightarrow{K}_{n}$ contains no rainbow triangles and $c(D)=\lfloor\frac{n^{2}}{4}\rfloor+1$, then $c(D-v)\leq \lfloor\frac{(n-1)^{2}}{4}\rfloor+1$ and $d^{s}(v)\geq \lfloor\frac{n}{2}\rfloor$ for every $v\in V(D)$. On the other hand, by Lemma \ref{lem2}, there is a vertex $v\in V(D)$ such that $d^{s}(v)\leq \lfloor\frac{n}{2}\rfloor$. So there is a vertex $v\in V(D)$ such that $d^{s}(v)=\lfloor\frac{n}{2}\rfloor$ and $c(D-v)=\lfloor\frac{(n-1)^{2}}{4}\rfloor+1$. By induction hypothesis, the vertex set of $D-v$ can be partitioned into two subsets $\{a_1,a_2,\ldots,a_{\lfloor\frac{n-1}{2}\rfloor}\}$ and $\{b_1,b_2,\ldots,b_{\lceil\frac{n-1}{2}\rceil}\}$ such that the spanning subdigraph $H$ of $D$ with $A(H)=\{a_ib_j|i=1,2,\ldots,\lfloor\frac{n-1}{2}\rfloor;j=1,2,\ldots,\lceil\frac{n-1}{2}\rceil\}$ (or $A(H)=\{b_ja_i|i=1,2,\ldots,\lfloor\frac{n-1}{2}\rfloor;j=1,2,\ldots,\lceil\frac{n-1}{2}\rceil\}$) is rainbow and all arcs in $A(D)\setminus A(H)$ are colored by a same new color $c$. By symmetry, we only discuss the case $A(H)=\{a_ib_j|i=1,2,\ldots,\lfloor\frac{n-1}{2}\rfloor;j=1,2,\ldots,\lceil\frac{n-1}{2}\rceil\}$. If $n$ is odd, then we divide the rest of the proof into two cases.

\setcounter{case}{0}
\begin{case}
$CN^{+}(v)\bigcap C^{s}(v)=\emptyset$.
\end{case}

If there is a vertex $b_j$ such that $C(b_jv)\in C^{s}(v)$. Without loss of generality, we can assume that $C(b_1v)\in C^{s}(v)$. Considering triangles $va_1b_1v$ and $va_2b_1v$, we have $C(va_1)=C(a_1b_1)$ and $C(va_2)=C(a_2b_1)$. Considering triangles $a_1b_2va_1$ and $a_2b_2va_2$, we have $$C(b_2v)\in\{C(a_1b_1),C(a_1b_2)\}\bigcap\{C(a_2b_1),C(a_2b_2)\}.$$ But $A(H)=\{a_ib_j|i=1,2,\ldots,\lfloor\frac{n-1}{2}\rfloor;j=1,2,\ldots,\lceil\frac{n-1}{2}\rceil\}$ is rainbow, a contradiction. So $C(b_jv)\notin C^{s}(v)$, for $j=1,2,\ldots,\lceil\frac{n-1}{2}\rceil$.
Thus $C^{s}(v)\subseteq \{C(a_1v),\ldots,C(a_{\lfloor\frac{n-1}{2}\rfloor}v)\}$. Since $d^{s}(v)=\lfloor\frac{n}{2}\rfloor=\lfloor\frac{n-1}{2}\rfloor$, we can see that $C^{s}(v)= \{C(a_1v),\ldots,C(a_{\lfloor\frac{n-1}{2}\rfloor}v)\}$. Considering the triangle set
$$\left\{vua_1v|u\in V(D)\setminus \{v,a_1\}\right\}\bigcup\{va_1a_2v\},$$
we have
$$C(\{vu|u\in V(D)\setminus \{v\}\})=\{c\}.$$
Considering triangles $va_1b_jv$ and $va_2b_jv$, for $j=1,2,\ldots,\lceil\frac{n-1}{2}\rceil$, we have
$$C(b_jv)\in \{C(a_1b_j),c\}\bigcap\{C(a_2b_j),c\}=\{c\}.$$
Let $v=b_{\lceil\frac{n}{2}\rceil}$. Then we can see that the spanning subdigraph $H$ of $D$ with $A(H)=\{a_ib_j|i=1,2,\ldots,\lfloor\frac{n}{2}\rfloor;j=1,2,\ldots,\lceil\frac{n}{2}\rceil\}$ is rainbow and all the remaining arcs are colored by a same new color $c$.

\begin{case}
$CN^{-}(v)\bigcap C^{s}(v)=\emptyset$.
\end{case}
By similar analysis, we can see that the spanning subdigraph $H$ of $D$ with $A(H)=\{a_ib_j|i=1,2,\ldots,\lceil\frac{n}{2}\rceil;j=1,2,\ldots,\lfloor\frac{n}{2}\rfloor\}$ is rainbow and all the remaining arcs are colored by a same new color $c$, where $v=a_{\lceil\frac{n}{2}\rceil}$.

If $n$ is even, then by similar analysis we can see that the spanning subdigraph $H$ of $D$ with  $A(H)=\{a_ib_j|i=1,2,\ldots,\lfloor\frac{n}{2}\rfloor;j=1,2,\ldots,\lceil\frac{n}{2}\rceil\}$ is rainbow and all the remaining arcs are colored by a same new color $c$, where $v=a_{\lfloor\frac{n}{2}\rfloor}$.

The proof is complete. \qed
\setcounter{case}{0}

\noindent\textbf{Proof of Theorem \ref{thm2}.}
Suppose the contrary. Let $D$ be a counterexample with the smallest number of vertices, and then with the smallest number of arcs.
\setcounter{claim}{0}
\begin{claim}\label{claim2.1}
$D$ contains two arcs $uv$ and $xy$ with a same color, where $xy\neq vu$.
\end{claim}
\begin{proof}
Recall that the maximum number of arcs among all digraphs of order $n$ without directed triangles is $\lfloor\frac{n^{2}}{2}\rfloor$ (see \cite{Sridharan}). If $c(D)\geq\lfloor\frac{n^{2}}{2}\rfloor+1$, then $D$ contains a rainbow triangle, a contradiction. So $c(D)\leq \lfloor\frac{n^{2}}{2}\rfloor$. Thus, we have
\begin{equation*}
a(D)-\lfloor\frac{n^{2}}{2}\rfloor\geq n(n-1)+\lfloor\frac{n^{2}}{4}\rfloor+2-2\lfloor\frac{n^{2}}{2}\rfloor
=
\left\{
\begin{aligned}
&\frac{n(n-4)}{4}+2>0, &n \text{~is~even};\\
&\frac{(n-1)(n-3)}{4}+2>0, &n \text{~is~odd}.
\end{aligned}
\right.
\end{equation*}
So $a(D)>\lfloor\frac{n^{2}}{2}\rfloor$. Namely, $D$ contains a directed triangle $\Delta$ and at least two arcs of $\Delta$ are colored by a same color. Note that two arcs of a triangle can only have one common end. So $D$ contains two arcs $uv$ and $xy$ with a same color, where $xy\neq vu$.
\end{proof}

\begin{claim}\label{claim2.2}
\[a(D)+c(D)=\begin{cases}
n(n-1)+\lfloor\frac{n^{2}}{4}\rfloor+3,&\text{$n=3, 4$};\\
n(n-1)+\lfloor\frac{n^{2}}{4}\rfloor+2,&
\text{$n\geq 5$ }.
\end{cases}\]
\end{claim}
\begin{proof}
By Claim \ref{claim2.1}, let $a_{1}$ and $a_{2}$ be two arcs with a same color. Then $a(D-a_{1})=a(D)-1$ and $c(D-a_{1})=c(D)$. If
\[a(D)+c(D)\geq\begin{cases}
n(n-1)+\lfloor\frac{n^{2}}{4}\rfloor+4,&\text{$n=3, 4$};\\
n(n-1)+\lfloor\frac{n^{2}}{4}\rfloor+3,&
\text{$n\geq 5$ },
\end{cases}\]
then
\[a(D-a_{1})+c(D-a_{1})\geq\begin{cases}
n(n-1)+\lfloor\frac{n^{2}}{4}\rfloor+3,&\text{$n=3, 4$};\\
n(n-1)+\lfloor\frac{n^{2}}{4}\rfloor+2,&
\text{$n\geq 5$ }.
\end{cases}\]
Note that $D-a_{1}$ contains no rainbow triangles either. Thus $D-a_{1}$ is a counterexample with fewer arcs, a contradiction.
\end{proof}

\begin{claim}\label{claim2.3}
For every $v\in V(D)$, we have
\[d(v)+d^{s}(v)\geq\begin{cases}
2(n-1)+\frac{n}{2}+1,&n \text{~is~even};\\
2(n-1)+\frac{n-1}{2}+1,&n \text{~is~odd~and~} n\neq 5;\\
10, &n=5.
\end{cases}\]
\end{claim}
\begin{proof}
Note that $a(D-v)=a(D)-d(v)$ and $c(D-v)=c(D)-d^{s}(v)$. If
\[d(v)+d^{s}(v)\leq\begin{cases}
2(n-1)+\frac{n}{2},&n \text{~is~even};\\
2(n-1)+\frac{n-1}{2},&n \text{~is~odd~and~} n\neq 5;\\
9, &n=5,
\end{cases}\]
then
$$a(D-v)+c(D-v)=a(D)+c(D)-\left(d(v)+d^{s}(v)\right)$$
\begin{equation*}
\geq\begin{cases}
(n-1)(n-2)+\lfloor\frac{(n-1)^{2}}{4}\rfloor+3,&\text{$n=4, 5$};\\
(n-1)(n-2)+\lfloor\frac{(n-1)^{2}}{4}\rfloor+2,&\text{$n\geq 6$ }.
\end{cases}
\end{equation*}
Note that $D-v$ does not contain a rainbow triangle. Thus $D-v$ is a counterexample with fewer vertices, a contradiction.
\end{proof}

\begin{claim}\label{claim2.4}
$\sum_{v\in V(D)}d^{s}(v)\leq 2c(D)-1$.
\end{claim}
\begin{proof}
Let $c$ be an arbitrary color in $C(D)$. Note that  each color $c$ can only be saturated by at most two vertices. So $\sum_{v\in V(D)}d^{s}(v)\leq 2c(D)$. Moreover, $c$ is saturated by exactly two vertices if and only if $c$ appears on only one arc or on a pair of arcs between two vertices. By Claim \ref{claim2.1}, $D$ contains two arcs $uv$ and $xy$ with a same color, where $xy\neq vu$. Thus, at least one color cannot be saturated by exactly two vertices. So $\sum_{v\in V(D)}d^{s}(v)\leq 2c(D)-1$.
\end{proof}

By Claims \ref{claim2.2}-\ref{claim2.4}, we can get that if $n\geq 6$  is even, then
\begin{equation}
2n(n-1)+\frac{n^{2}}{2}+n\leq \sum_{v\in V(D)}\left(d(v)+d^{s}(v)\right)\leq 2a(D)+2c(D)-1=2n(n-1)+\frac{n^{2}}{2}+3.
\end{equation}
This implies that $n\leq 3$, a contradiction.

If $n\geq 7$ is odd, then
\begin{equation}
\begin{aligned}
2n(n-1)+\frac{n(n-1)}{2}+n\leq \sum_{v\in V(D)}\left(d(v)+d^{s}(v)\right)&\leq 2a(D)+2c(D)-1\\
&=2n(n-1)+\frac{(n-1)(n+1)}{2}+3.
\end{aligned}
\end{equation}
This implies that $n\leq 5$, a contradiction. So it suffices to consider the cases $n=3, 4, 5$.

For $n=3$, since $a(D)+c(D)= 11$ and $a(D)\leq 6$, we have $c(D)\geq 5=\lfloor\frac{n^{2}}{2}\rfloor+1$. So $D$ contains a rainbow triangle, a contradiction.

For $n=4$, we have $a(D)+c(D)\geq 19$. If $a(D)=12$, then $D\cong \overleftrightarrow{K}_4$ and $c(D)\geq 7$. By Theorem \ref{thm1}, $D$ contains a rainbow triangle, a contradiction. If $a(D)\leq 10$, then $c(D)\geq 9=\lfloor \frac{4^2}{2}\rfloor +1$. We know that $D$ contains a rainbow triangle, a contradiction. The only case left is that $a(D)=11=a(\overleftrightarrow{K}_4)-1$ and $c(D)=8$. Let $u$ be a vertex in $D$ such that $D-u\cong \overleftrightarrow{K}_3$. Since $f(\overleftrightarrow{K}_3)=5$, we have $d^s(u)\geq 4$. Let $V(D-u)=\{x,y,z\}$. Then there must exist two vertices in $V(D-u)$ (say $x$ and $y$) such that $c(ux)$ and $c(yu)$ are two distinct colors in $C^s(u)$. This implies that $uxyu$ is a rainbow triangle, a contradiction.

\begin{lem}\label{lem3}
Let $D$ be an arc-colored digraph of order $3$. If $a(D)+c(D)=10$ and $D$ contains no rainbow triangle, then $D\cong\overleftrightarrow{K}_{3}$.
\end{lem}
\begin{proof}
Since $D$ contains no rainbow triangle, we have $c(D)\leq \lfloor\frac{n^{2}}{2}\rfloor=4$ and $a(D)\geq 6$. So $c(D)=4$, $a(D)=6$ and $D\cong\overleftrightarrow{K}_{3}$.
\end{proof}
%
%
\begin{lem}\label{lem4}
	Let $D$ be an arc-colored digraph of order $4$. If $a(D)+c(D)=18$ and $D$ contains no rainbow triangle, then $D\cong\overleftrightarrow{K}_{4}$.
\end{lem}
\begin{proof}
For every $v\in V(D)$, since $D-v$ contains no rainbow triangles, we have $a(D-v)+c(D-v)\leq 10$ and hence $d(v)+d^{s}(v)\geq 8$. If $d(v)+d^s(v)\geq 9$ for every $v\in V(D)$, then
\begin{equation}
36\leq \sum_{v\in V(D)}\left(d(v)+d^{s}(v)\right)\leq 2a(D)+2c(D)-1=35,
\end{equation}
a contradiction. So there is a vertex $v\in V(D)$ such that $d(v)+d^{s}(v)= 8$. Let $V(D)=\{v, x, y, z\}$ and $d(v)+d^{s}(v)= 8$. Then $a(D-v)+c(D-v)= 10$. By Lemma \ref{lem3}, $D-v\cong\overleftrightarrow{K}_{3}$, and thus $D-v\in\mathcal{G}_3$.
Furthermore, by Theorem \ref{thm1+}, we know that the color sets of the two directed triangles in $D-v$ is disjoint. Let $C(D-v)=\{1, 2, 3, 4\}$. If $D\not\cong \overleftrightarrow{K}_{4}$, then $d(v)\leq 5$ and $d^{s}(v)\geq 3$. Let $\{5,6,7\}\subseteq C^{s}(v)$. If there exist two vertices in $V(D-v)$ (say $x$ and $y$) such that $c(vx)$ and $c(yv)$ are two distinct colors in $C^s(v)$, then we have $vxyv$ is a rainbow triangle, a contradiction. So we can assume that $C(vx)=5, C(vy)=6$ and $C(vz)=7$. If  $yv\in A(D)$, then consider triangles $vxyv$ and $vzyv$. We get $C(xy)=C(yv)$ and $C(zy)=C(yv)$. Thus $C(xy)=C(zy)$. This contradicts the structure of $D-v\in\mathcal{G}_3$. So we have $yv\not\in A(D)$. Similarly, we can get $xv,zv\not\in A(D)$. Thus $d(v)=d^s(v)=3$. This contradicts that $d(v)+d^{s}(v)= 8$.
%
%
\end{proof}

For $n=5$, we have $a(D)+c(D)\geq 28$. For each integer $p$, let $X_p=\{u\in V(D):a(D-u)+c(D-u)=p\}$ and let $x_p=|X_p|$. Since $D$ contains no rainbow triangle, $a(D-u)+c(D-u)\leq 18$ for each vertex $u\in V(D)$. So we have
\begin{equation}\label{eq1}
  \sum_{p\leq 18} x_p=5.
\end{equation}
Let $Y_i=\{u: i\in C(D-u)\}$ for each $i\in C(D)$ and let $y_i=|Y_i|$. Since each color appears in at least $3$ induced subdigraphs of order $4$, we have $y_i\geq 3$.
Note that $D$ has $5$ induced subdigraphs of order $4$, every arc of $D$ belongs to exactly $3$ of such induced subdigraphs and every color $i\in C(D)$ belongs to exactly $y_{i}$ of them. So we have

\begin{equation}\label{eq2}
\sum_{p\leq 18}px_p=3a(D)+\sum_{i\in C(D)}y_i=3a(D)+3c(D)+\sum_{i\in C(D)}(y_i-3)\geq 84+\sum_{i\in C(D)}(y_i-3).
\end{equation}
By $(\ref{eq2})-16\times(\ref{eq1})$ we can get $$\sum_{i\in C(D)}(y_i-3)\leq 2x_{18}+x_{17}-4.$$
\begin{case}
	$x_{18}=0$.
\end{case}
In this case, since $x_{17}\leq 5$, we have $0\leq \sum_{i\in C(D)}(y_i-3)\leq 1$. This means that either $y_i=3$ for all $i\in C(D)$ or there is only one color $j$ such that $y_j=4$.

If $y_i=3$ for all $i\in C(D)$, then every triangle in $D$ must be a rainbow triangle. This implies that $D$ contains no directed triangles. So $a(D)\leq \lfloor \frac{5^2}{2}\rfloor=12$. Thus $$28\leq a(D)+c(D)\leq 2a(D)\leq 24,$$
a contradiction.  If there is only one color $j$ such that $y_j=4$. Then let $u$ be the only vertex in $D$ such that $j\not\in C(D-u)$. Then $D-u$ contains no directed triangle. Thus $a(D-u)+c(D-u)\leq 2a(D-u)\leq 2\lfloor \frac{4^2}{2}\rfloor=16$. So $d^s(u)+d(u)\geq 12$. Note that $d^s(u)+d(u)\leq 2d(u)-a(D^j)+1$. So
\begin{equation}\label{eqt5}
a(D^j)\leq 2d(u)-11.
\end{equation}

On the other hand, let $D'$ be an arc-colored digraph such that $V(D')=V(D)$ and $A(D')=(A(D)\backslash A(D^j))\cup \{e\}$. Here $e$ is an arc from $D^j$. Then we have $28-a(D^j)+1=a(D')+c(D')\leq 2a(D')\leq 2\lfloor \frac{5^2}{2}\rfloor.$
Thus
\begin{equation}\label{eqt6}
a(D^j)\geq 5.
\end{equation}
Combine (\ref{eqt5}) and (\ref{eqt6}). We have $d(u)\geq 8$. Note that $d(u)\leq 8$. We have $d(u)=8$,  $a(D^j)=5$ and there must be a vertex $v\in V(D-u)$ such that $C(uv)=C(vu)=j$. Let $D''$ be an arc-colored digraph such that $V(D'')=V(D)$ and $A(D'')=(A(D)\backslash A(D^j))\cup \{uv,vu\}$. Then each triangle in $D''$ must be a rainbow triangle. So $D''$ contains no triangles. We have
$$a(D)-a(D^j)+2=a(D'')\leq \lfloor \frac{5^2}{2}\rfloor.$$
Thus $a(D)\leq 15$. So $c(D)\geq 13=\lfloor \frac{5^2}{2}\rfloor+1$, which implies that $D$ contains a rainbow triangle, a contradiction.

\begin{case}
	$x_{18}\geq 1$.
\end{case}
 In this case,  there is a vertex $u\in V(D)$ such that $a(D-u)+c(D-u)= 18$ and $d(u)+
 d^{s}(u)\geq 10$. By Lemma \ref{lem4}, we can see that $D-u\cong\overleftrightarrow{K}_{4}$ and $D-u\in \mathcal{G}_4$.  If $D\cong \overleftrightarrow{K}_5$, then we obtain a rainbow triangle by Theorem \ref{thm1}, a contradiction. So $d(u)\leq 7$  and $d^s(u)\geq 3$. By Lemma \ref{lem1}, we can assume that $CN^-(u)\cap C^s(u)=\emptyset$. Then $d^s(u)\leq 4$. Let the two monochromatic cycles in $D-u$ are $xyzwx$ and $wzyxw$ with colors $\alpha$ and $\beta$, respectively.
 Assume that $C(ux), C(uy)$ and $C(uz)$ are three distinct colors in $C^s(u)$. If $yu\in A(D)$, then consider triangles $uxyu$ and $uzyu$, we get $\alpha=C(yu)=\beta$, a contradiction. So $yu\not\in A(D)$. Similarly, we can get $xu\notin A(D)$, $zu\notin A(D)$, $wu\notin A(D)$. So $d(u)\leq 4$, and thus $d(u)+d^s(u)\leq 8$, a contradiction. 

The proof is complete.\qed

To prove Theorem \ref{thm3}, we need the following famous theorem of Moon \cite{Moon: 1966}:
\begin{thm}[Moon's theorem]
Let $T$ be a strongly connected tournament on $n\geq 3$ vertices. Then each vertex of $T$ is contained in a cycle of length $k$ for all $k\in [3,n]$. In particular, a tournament is hamiltonian if and only if it is strongly connected.
\end{thm}

\noindent\textbf{Proof of Theorem \ref{thm3}.}
By induction on $n$. For $n=3$, since $D$ is strongly connected, we can see that $D$ is a directed triangle. If $c(D)\geq \frac{n(n-1)}{2}-n+3=3$, then all arcs of $D$ have distinct colors. So $D$ is a rainbow triangle.

Suppose that every arc-colored strongly connected tournament $D'$ of order $n-1$ with $c(D')\geq \frac{(n-1)(n-2)}{2}-(n-1)+3$ contains a rainbow triangle for $n\geq 4$. Now we consider an arc-colored strongly connected tournament $D$ of order $n$. Since $D$ is strongly connected, by Moon's theorem, $D$ contains a directed $(n-1)$-cycle $C$. Let $v$ be the vertex not in $C$. Then $D-v$ contains a hamiltonian cycle $C$. Thus, $D-v$ is strongly connected. If $c(D)\geq \frac{n(n-1)}{2}-n+3$ and $D$ contains no rainbow triangles, then $D-v$ contains no rainbow triangles either, and hence $c(D-v)\leq \frac{(n-1)(n-2)}{2}-(n-1)+2$. So we have
\begin{equation*}
d^{s}(v)\geq \frac{n(n-1)}{2}-n+3-\left(\frac{(n-1)(n-2)}{2}-(n-1)+2\right)=n-1.
\end{equation*}
This implies that $CN(v)\bigcap C(D-v)=\emptyset$ and every two different arcs incident to $v$ have distinct colors. Since $D$ is strongly connected, there exists an arc from $N^{+}(v)$ to $N^{-}(v)$. Assume that $wu\in A(D)$, where $w\in N^{+}(v)$ and $u\in N^{-}(v)$, then $vwuv$ is a directed triangle. Since $wu\in A(D-v)$ and $vw$, $uv$ are two different arcs incident to $v$, we can see that $vwuv$ is a rainbow triangle, a contradiction.

The proof is complete. \qed

\section{Concluding remarks}
By Lemmas \ref{lem3} and \ref{lem4} in Theorem \ref{thm2}, we proved that for $n=3$, $4$, if $a(D)+c(D)=a(\overleftrightarrow{K}_{n})+f(\overleftrightarrow{K}_{n})-1$ and $D$ contains no rainbow triangles, then $D\cong\overleftrightarrow{K}_{n}$. We conjecture that this is true for all $n\geq 5$.

\begin{conj}
Let $D$ be an arc-colored digraph of order $n\geq 5$ without containing rainbow triangles. If
$
a(D)+c(D)=n(n-1)+\lfloor\frac{n^{2}}{4}\rfloor+1,
$
then $D\cong\overleftrightarrow{K}_{n}$.
\end{conj}


\begin{thebibliography}{10}

\bibitem{Albert: 1995}
{M. Albert, A. Frieze and B. Reed},
\newblock Multicolored Hamilton cycles,
\newblock {\em Electron. J. Combin.,} {\bf 2} (1995),  \#R10.

\bibitem{Bang-Jensen: 2001}
{J. Bang-Jensen and G. Gutin},
\newblock Digraphs: Theory, Algorithms and Applications,
\newblock{\em Springer}, 2001.


\bibitem{Bondy: 2008}
{J.A. Bondy and U.S.R. Murty},
\newblock {\em Graph Theory},
\newblock {Springer}, 2008.

\bibitem{Erdos: 1983}
{P. Erd\H{o}s, J. Ne\v{s}et\v{r}il and V. R\"{o}dl},
\newblock Some problems related to partitions of edges of a graph,
\newblock {\em Graphs and other combinatorial topics, Teubner, Leipzig }, (1983) 54--63.

\bibitem{Frieze: 1993}
{A.M. Frieze and B.A. Reed},
\newblock Polychromatic Hamilton cycles,
\newblock {\em Discrete Math.,} {\bf 118} (1993) 69--74.


\bibitem{FLZ: 2017}
{S. Fujita, R. Li and S. Zhang},
\newblock Color degree and monochromatic degree conditions for short properly colored cycles in edge-colored graphs,
\newblock {\em J. Graph Theory, } {\bf 87} (2018) 362--373.

\bibitem{Fujita: 2014}
{S. Fujita, C. Magnant and K. Ozeki},
\newblock Rainbow generalizations of Ramsey theory: a survey,
\newblock{\em Graphs Combin.,} {\bf 26} (2010) 1--30.

\bibitem{Gallai: 1967}
{T. Gallai},
\newblock Transitiv orientierbare Graphen,
\newblock {\em Acta Math. Hungar.,} {\bf 18} (1967) 25--66.

\bibitem{Gyarfas: 2004}
{A. Gy\'{a}rf\'{a}s and G. Simonyi},
\newblock Edge colorings of complete graphs without tricolored triangles,
\newblock{\em J. Graph Theory,} {\bf 46} (2004) 211--216.

\bibitem{Hahn: 1986}
{G. Hahn and C. Thomassen},
\newblock Path and cycle sub-Ramsey numbers and an edge-colouring conjecture,
\newblock{\em Discrete Math.,} {\bf 62 (1)} (1986) 29--33.



\bibitem{X.Li: 2008}
{M. Kano and X. Li},
\newblock Monochromatic and heterochromatic subgraphs in edge-colored graphs - a survey,
\newblock {\em Graphs Combin.,} {\bf 24} (2008) 237--263.

\bibitem{B.Li: 2014}
{B. Li, B. Ning, C. Xu and S. Zhang},
\newblock Rainbow triangles in edge-colored graphs,
\newblock {\em European J. Combin.,} {\bf 36} (2014) 453--459.

\bibitem{H.Li: 2013}
{H. Li},
\newblock Rainbow $C_3$'s and $C_4$'s in edge-colored graphs,
\newblock {\em Discrete Math.,} {\bf 313} (2013) 1893--1896.

\bibitem{Li-Wang: 2012}
{H. Li and G. Wang},
\newblock Color degree and heterochromatic cycles in edge-colored graphs,
\newblock {\em European J. Combin.,} {\bf 33} (2012) 1958--1964.

\bibitem{Li-Ning-Zhang: 2016}
{R. Li, B. Ning and S. Zhang},
\newblock Color degree sum conditions for rainbow triangles in edge-colored graphs,
\newblock {\em Graphs Combin.,} {\bf 32} (2016) 2001--2008.

\bibitem{Li-Zhang-Bai-Li: 2018}
{W. Li, S. Zhang, Y. Bai and R. Li},
\newblock Rainbow triangles in arc-colored tournaments,
\newblock  arXiv:1805.03412.

\bibitem{Moon: 1966}
{J.W. Moon},
\newblock On subtournaments of a tournament.
\newblock {\em Canad. Math. Bull.,} {\bf 9} (1966) 297--301.

\bibitem{Sridharan}
{Sridharan, M. R.},
\newblock On Tur$\acute{a}$n's theorem,
\newblock {\em J. Indian Inst. Sci.,} {\bf 59 (10)} (1977) 111--112.
\end{thebibliography}
\end{document}